\documentclass[10pt]{article}
\usepackage{amsmath}
\usepackage{amsfonts}
\usepackage{amsthm}
\usepackage{amssymb}
\usepackage{color}

\newtheorem{theorem}{Theorem}[section]

\newtheorem{lemma}[theorem]{Lemma}
\newtheorem{remark}[theorem]{Remark}

\numberwithin{equation}{section}

\title{Existence result for differential inclusion with $p(x)$-Laplacian}

\author{
  Sylwia Barna\'s\\
  email: Sylwia.Barnas@im.uj.edu.pl\\
\\
  Cracow University of Technology\\
  Institute of Mathematics\\
  ul. Warszawska 24,\
  31-155 Krak\'ow, Poland\\
\\
  Jagiellonian University\\
  Department of Mathematics and Computer Science\\
  ul. {\L}ojasiewicza 6,\
  30-348 Krak\'ow, Poland}

\begin{document}
\maketitle
\bibliographystyle{plain}


\noindent \textbf{Abstract:}
  In this paper we study the nonlinear elliptic problem with
  $p(x)$-Laplacian (hemivariational inequality).
  We prove the existence of a nontrivial solution.
  Our approach is based on critical point theory for
  locally Lipschitz functionals due to Chang \cite{changg}.

\section{Introduction} \label{Intr}
  Let $\Omega\subseteq\mathbb{R}^N$ be a bounded domain with a $\mathcal{C}^2$-boundary $\partial \Omega$ and $N>2$.
  In this paper we consider a differential inclusion in $\Omega$ involving a $p(x)$-Laplacian of the type

\[
  \left\{
  \begin{array}{lr}
    -\Delta_{p(x)}u-\lambda |u(x)|^{p(x)-2} u(x) \in \partial j(x, u(x))& \textrm{in  } \Omega, \tag{P}\\
    u=0 & \textrm{on}\ \partial \Omega,
  \end{array}
  \right.
\]
  where $p:\overline{\Omega} \rightarrow \mathbb{R}$ is a continuous function satisfying
\[
  1<p^-:=\inf\limits_{x \in \Omega} p(x) \leqslant p(x) \leqslant  p^+:=\sup\limits_{x \in \Omega} p(x)<N<\infty
\]
 and 
\[
  p^+\leqslant \widehat{p}^*:=\frac{Np^-}{N-p^-}.
\]
  A functional $j(x,t)$ is a measurable in the first variable and locally Lipschitz
  in the second variable.
  By $\partial j(x,t)$ we denote the subdifferential of $j(x,\cdot)$ in the sense of Clarke \cite{Clarke}. The operator 
\[
  \Delta_{p(x)}u:= \textrm{div} \big( |\nabla u(x)|^{p(x)-2} \nabla u(x) \big)
\]
  is the so-called $p(x)$-Laplacian, which becomes $p$-Laplacian when $p(x)\equiv p$.
  Problems with $p(x)$-Laplacian are more complicated than with $p$-Laplacian,
  in particular, they are inhomogeneous and possesses "more nonlinearity".

  In our problem (P) a parametr $\lambda$ appears, for which we will assume that
  $\lambda<\frac{p^-}{p^+}\lambda_*$, where $\lambda_*$ is defined by
\[
  \lambda_*=\inf_{u \in W_0^{1,p(x)}(\Omega)\backslash \{0\}}
  \frac{\int_{\Omega}|\nabla u(x)|^{p(x)}dx}{\int_{\Omega}|u(x)|^{p(x)}dx}.
\]
  It may happen that $\lambda_*=0$ (see Fan-Zhang-Zhao \cite{fanzhangzhao}).\\

More recently, the existence of solutions for variable exponent differential inclusions with different boundary value conditions have been widely investigated by many authors, which are usually reduced to the solutions of Dirichlet and Neumann type problems. 
  For instance, we have papers where hemivariational inequalities involving $p(x)$-Laplacian
  are studied. 
  In Ge-Xue \cite{ge} and Qian-Shen \cite{qian1}, differential inclusions involving $p(x)$-Laplacian
  and Clarke subdifferential with Dirichlet boundary condition is considered.
  In the last paper the existence of two solutions of constant sign is proved.
  Hemivariational inequalities with Neumann boundary condition were studied in
  Qian-Shen-Yang \cite{qian}
  and Dai \cite{dai}.
  In Qian-Shen-Yang \cite{qian}, the inclusions involve a weighted function
  which is indefinite. In Dai \cite{dai}, the existence of infinitely many nonnegative solutions
  is proved. 
All the above mentioned papers deal with the so-called hemivariational inequalities,
  i.e. the multivalued part is provided by the Clarke subdifferential of the nonsmooth potential
  (see e.g. Naniewicz-Panagiotopoulos \cite{nana}).\\

 Our starting point for considering problems with $p(x)$-Laplacian were the papers of Gasi\'nski-Papageorgiou \cite{lg1,lg3,lg5} and
  Kourogenic-Papageorgiou \cite{ko}, where the authors deal with the constant exponent problems i.e. when $p(x)=p$. Moreover, the similar kind of problems were considered in D'Agu\`{i} - Bisci \cite{1} and Marano - Bisci - Motreanu \cite{2}. In the first of this paper, the authors deal with a perturbed eigenvalue Dirichlet-type problem for an elliptic hemivariational
inequality involving the p-Laplacian. In the last paper, the existence of
multiple solutions is investigated  by the use of classical techniques due
to Struwe and a recent saddle
point theorem for locally Lipschitz continuous functionals. \\

  In a recent paper (see Barna\'s \cite{barnas}), we examined nonlinear hemivariational inequality with $p(x)$-Laplacian. We proved an existence theorem under the assumptions that the Clarke subdifferential of the generalized potential is bounded. 
  In the present paper, this hypothesis is more general and we assume the so-called sub-critical growth condition.
  Moreover, we also take more general assumption about behaviour in the neighbourhood of $0$.

  The techniques of this paper differ from these used in the above mentioned paper. 
  Our method is more direct and general.\\

  Hemivariational inequalities arise in physical problems when we deal with 
  nonconvex nonsmooth energy functionals. Such functions appear quite often in 
  mechanics and engineering. 
  For instance, we have a lot of applications to modelling electrorheological fluids (see Ru$\check{\textrm{z}}$i$\check{\textrm{c}}$ka \cite{ruzicka}) and image restoration. \\ 

\section{Mathematical preliminaries} \label{Prelim}

  Let $X$ be a Banach space and $X^*$ its topological dual.
  By $\|\cdot\|$ we will denote the norm in $X$ and by
  $\langle\cdot,\cdot\rangle$ the duality brackets for the pair $(X,X^*)$.
  In analogy with the directional derivative of a convex function,
  we define the generalized directional derivative of a locally Lipschitz function $f$ at $x \in X$ in the direction $h \in X$
  by
\[
    f^0 (x;h)=\limsup_{x' \rightarrow 0, \lambda \rightarrow 0} \frac{f(x+x'+\lambda h)-f(x+x')}{\lambda}.
\]
  The function $h \longmapsto f^0 (x,h) \in \mathbb{R}$ is sublinear,
  continuous so it is the support function of a nonempty, convex and $w^*$-compact set
\[
    \partial f(x)=\{x^* \in X^*: \langle x^*,h \rangle \leqslant f^0 (x,h) \textrm{ for all } h \in X\}.
\]
  The set $\partial f(x)$ is known as generalized or Clarke subdifferential of $f$ at $x$. If $f$ is strictly differentiable at $x$ (in particular if $f$ is continuously G$\hat{\textrm{a}}$teaux differentiable at $x$), then $\partial f(x)=\{f'(x)\}$.

 Let $f:X\rightarrow\mathbb{R}$ be a locally Lipschitz function.
  From convex analysis it is well know that a proper, convex and lower
  semicontinuous function
  $g: X \rightarrow \overline{\mathbb{R}}=\mathbb{R}\cup \{ + \infty\}$
  is locally Lipschitz in the interior of its effective domain $\textrm{dom} g=\{x \in X:g(x)< \infty\}$.

  A point $x \in X$ is said to be a critical point of the locally
  Lipschitz function $f: X \rightarrow \mathbb{R}$, if $0 \in \partial f(x)$.
  If $x \in X$ is local minimum or local maksimum of $f$, then $x$ is critical point, and moreover this time the value $c=f(x)$ is called a critical value of
  $f$. From more details on the generalized subdifferential we refer to Clarke \cite{Clarke}, Gasi\'nski-Papageorgiou \cite{lg3}, Motreanu-Panagiotopoulos \cite{mon} and Motreanu-Radulescu \cite{mon2}.

  The critical point theory for smooth functions uses a compactness type
  condition known as the Palais-Smale condition. In our nonsmooth setting, the
  condition takes the following form. We say that locally Lipschitz function $f:X\rightarrow \mathbb{R} $ satisfies the nonsmooth Palais-Smale condition (nonsmooth PS-condition for short),
  if any sequence $\{x_n\}_{n \geqslant 1} \subseteq X$ such that $\{f(x_n)\}_{n \geqslant 1}$
  is bounded and $m(x_n)=\min\{ \|x^*\|_* :x^* \in \partial f(x_n)\}\rightarrow 0$
  as $n \rightarrow \infty$, has a strongly convergent subsequence.

  The first theorem is due to Chang \cite{changg} and extends to
  a nonsmooth setting the well known "mountain pass theorem" due
  to Ambrosetti -Rabinowitz \cite{ambro} (see also Radulescu \cite{rad}).

\begin{theorem} \label{twierdzenie}
  If $X$ is a reflexive Banach space, $R:X \rightarrow \mathbb{R}$ is
  a locally Lipschitz functional satisfying PS-condition and for some
  $\rho>0$ and $y \in X$ such that $\|y\|>\rho$, we have
\[
    \max\{R(0),R(y)\}<\inf\limits_{\|x\|=\rho} \{R(x)\}=: \eta,
\]
   then R has a nontrivial critical point $x \in X$ such that
   the critical value $c=R(x) \geqslant \eta$ is characterized by the following minimax principle
\[
    c=\inf\limits_{\gamma \in \Gamma}\max\limits_{0 \leqslant \tau \leqslant 1}\{R(\gamma(\tau))\},
\]
  where $\Gamma=\{\gamma \in \mathcal{C}([0,1],X):\gamma(0)=0,\gamma(1)=y\}$.
\end{theorem}

  In order to discuss problem (P),
  we need some theories on the spaces $L^{p(x)}(\Omega)$ and $W^{1,p(x)}(\Omega)$,
  which we call generalized Lebesgue-Sobolev spaces (see Fan-Zhao \cite{fan,orlicz} and Kov$\acute{\textrm{a}}\check{\textrm{c}}$ik-R$\acute{\textrm{a}}$kosnik \cite{kovacik}).

  By $S(\Omega)$ we denote the set of all measurable real-valued function defined on $\mathbb{R}^N$.
  We define
\[
    L^{p(x)}(\Omega)=\{u \in S(\Omega): \int_{\Omega} |u(x)|^{p(x)} dx<\infty \}.
\]
  We furnish $L^{p(x)}(\Omega)$ with the following norm (known as the Luxemburg norm)
\[
    \|u\|_{p(x)}=\|u\|_{L^{p(x)}(\Omega)}=\inf \Big\{\lambda>0: \int_{\Omega} \Big|\frac{u(x)}{\lambda}\Big|^{p(x)}dx \leqslant 1 \Big\}.
\]
 Also we introduce the variable exponent Sobolev space
\[
    W^{1,p(x)}(\Omega) = \{u \in L^{p(x)}(\Omega):|\nabla u| \in L^{p(x)}(\Omega)\},
\]
  and we equip it with the norm
\[
    \|u\|=\|u\|_{W^{1,p(x)}(\Omega)} = \|u\|_{p(x)}+\|\nabla u\|_{p(x)}.
\]
  By $W_0^{1,p(x)}(\Omega)$ we denote the closure of $C_0^\infty (\Omega)$ in $W^{1,p(x)}(\Omega)$.

\begin{lemma} [Fan-Zhao \cite{fan}]  \label{fan}
 If $\Omega \subset \mathbb{R}^N$ is an open domain, then

\noindent (a) the spaces $L^{p(x)}(\Omega)$, $W^{1,p(x)}(\Omega)$ and $W_0^{1,p(x)}(\Omega)$ are
  separable and reflexive Banach spaces;

\noindent
 (b) the space $L^{p(x)}(\Omega)$ is uniformly convex;

\noindent
 (c) if $1 \leqslant q(x) \in \mathcal{C}(\overline{\Omega})$ and $q(x) \leqslant p^*(x)$ (respectively $q(x) < p^*(x)$) for any $x \in \overline{\Omega}$, where
\[
  p^*(x)=\left\{
  \begin{array}{ll}
   \frac{Np(x)}{N-p(x)} & p(x)<N\\
    \infty & p(x) \geqslant N,
  \end{array}
  \right.
\]
 then $W^{1,p(x)}(\Omega)$ is embedded continuously (respectively compactly) in $L^{q(x)}(\Omega)$;

\noindent
 (d) Poincar\'e inequality in $W_0^{1,p(x)}(\Omega)$ holds i.e., there exists a positive constant $c$ such that
 \[
  \|u\|_{p(x)} \leqslant c \|\nabla u\|_{p(x)}  \qquad \textrm{for all } u \in W_0^{1,p(x)}(\Omega);
\]

\noindent
 (e) $(L^{p(x)}(\Omega))^*=L^{p'(x)}(\Omega)$, where $\frac{1}{p(x)}+\frac{1}{p'(x)}=1$
   and for all $u \in L^{p(x)}(\Omega)$ and $v \in L^{p'(x)}(\Omega)$, we have
\[
  \int_{\Omega}|uv|dx\leqslant \Big(\frac{1}{p^-}+\frac{1}{
  {p'}^{-}}\Big)\|u\|_{p(x)}\|v\|_{p'(x)}.
\]
\end{lemma}

\begin{lemma}[Fan-Zhao \cite{fan}]\label{lemma2}
  Let $\varphi (u)=\int_{\Omega} |u(x)|^{p(x)} dx$ for $u \in L^{p(x)}(\Omega)$ and let
  $\{u_n\}_{n \geqslant 1} \subseteq L^{p(x)}(\Omega)$.

\noindent (a) for $u\neq 0$, we have
\begin{center} $\|u\|_{p(x)} =a \Longleftrightarrow \varphi (\frac{u}{a})=1$;\end{center}

\noindent (b) we have
\begin{center}$\|u\|_{p(x)}<1 (respectively\; =1, >1) \; \Longleftrightarrow \; \varphi(u)<1 (respectively\; =1, >1)$;\end{center}

\noindent (c) if $\|u\|_{p(x)}>1$, then
\begin{center}$\|u\|^{p^-}_{p(x)} \leqslant \varphi (u) \leqslant \|u\|^{p^+}_{p(x)}$;\end{center}

\noindent (d) if $\|u\|_{p(x)}<1$, then
\begin{center}$\|u\|^{p^+}_{p(x)} \leqslant \varphi (u) \leqslant \|u\|^{p^-}_{p(x)}$;\end{center}

\noindent (e) we have
\begin{center}$\lim\limits_{n \rightarrow \infty} \|u_n\|_{p(x)} =0 \; \Longleftrightarrow \; \lim\limits_{n \rightarrow \infty} \varphi(u_n) =0$;\end{center}

\noindent (f) we have
\begin{center}$\lim\limits_{n \rightarrow \infty}\|u_n\|_{p(x)} = \infty \; \Longleftrightarrow \; \lim\limits_{n \rightarrow \infty} \varphi (u_n) =\infty$.\end{center}
\end{lemma}

  Similarly to Lemma \ref{lemma2}, we have the following result.

\begin{lemma} [Fan-Zhao \cite{fan}] \label{lemma5}
  Let $\Phi (u) = \int_{\Omega} (|\nabla u(x)|^{p(x)}+|u(x)|^{p(x)})dx$
  for $u \in W^{1,p(x)}(\Omega)$ and let $\{u_n\}_{n \geqslant 1} \subseteq W^{1,p(x)}(\Omega)$.
  Then

\noindent (a) for $u \neq 0$, we have
\begin{center}$\|u\|=a \; \Longleftrightarrow \; \Phi (\frac{u}{a})=1$;\end{center}

\noindent (b) we have
\begin{center}$\|u\|<1 (respectively\; =1, >1) \; \Longleftrightarrow \; \Phi (u) <1 (respectively\; =1, >1) $;\end{center}

\noindent (c) if $\|u\|>1$, then
\begin{center} $\|u\|^{p^-} \leqslant \Phi (u) \leqslant  \|u\|^{p^+}$;\end{center}

\noindent (d) if $\|u\|<1$, then
\begin{center}$\|u\|^{p^+} \leqslant \Phi (u) \leqslant  \|u\|^{p^-}$;\end{center}

\noindent (e) we have
\begin{center}$\lim\limits_{n \rightarrow \infty} \|u_n\|=0 \Longleftrightarrow \lim\limits_{n \rightarrow \infty} \Phi (u_n) =0$;\end{center}

\noindent (f) we have
\begin{center}$\lim\limits_{n \rightarrow \infty} \|u_n\| = \infty \Longleftrightarrow \lim\limits_{n \rightarrow \infty} \Phi (u_n) = \infty$.\end{center}
\end{lemma}

  Consider the following function
\[
    J(u)=\int_{\Omega} \frac{1}{p(x)}|\nabla u|^{p(x)} dx, \qquad \textrm{for all } u \in W_0^{1,p(x)}(\Omega).
\]
  We know that $J \in \mathcal{C}^1 (W_0^{1,p(x)}(\Omega))$ and operator $-\textrm{div}(|\nabla u|^{p(x)-2} \nabla u)$,
  is the derivative operator of $J$ in the weak sense (see Chang \cite{chang}).
  We denote
\[
  A=J': W_0^{1,p(x)}(\Omega) \rightarrow (W_0^{1,p(x)}(\Omega))^*,
\]
  then
\begin{equation}\label{wazne}
  \langle Au,v\rangle = \int_{\Omega} |\nabla u(x)|^{p(x)-2} (\nabla u(x), \nabla v(x)) dx,
\end{equation}
  for all $u, v \in W_0^{1,p(x)}(\Omega)$.

\begin{lemma}[Fan-Zhang \cite{zhang}] \label{lemma3}
  If A is the operator defined above, then $A$ is a continuous,
  bounded, strictly monotone and maximal monotone operator of type $(S_+)$ i.e.,
  if $u_n \rightarrow u$ weak in $W_0^{1,p(x)}(\Omega)$ and
  $\limsup\limits_{n \rightarrow \infty} \langle Au_n, u_n -u \rangle \leqslant 0$,
  implies that $u_n \rightarrow u$ in  $W_0^{1,p(x)}(\Omega)$.
\end{lemma}

\section{Existence of Solutions}
  We start by introducing our hypotheses on the function $j(x,t)$.

\noindent $H(j) \; j:\Omega  \times \mathbb{R} \rightarrow \mathbb{R}$ is a function such that $j(x,0)=0$ on $\Omega$ and

\noindent \textbf{(i)} for all $x \in \mathbb{R}$, the function $\Omega \ni x \rightarrow j(x,t) \in \mathbb{R}$ is measurable;

\noindent \textbf{(ii)} for almost all $x \in \Omega$, the function $\Omega \ni t \rightarrow j(x,t) \in \mathbb{R}$ is locally Lipschitz;

\noindent \textbf{(iii)} for almost all $x \in \Omega$ and all $v \in \partial j(x,t)$, we have $|v| \leq a(x)+c_1|t|^{r(x)-1}$ with $a \in L_+^\infty (\Omega), c_1>0$ and $r \in \mathcal{C}(\overline{\Omega})$, such that $p^+ \leq r^+:=\max\limits_{x \in \Omega} r(x)<\widehat{p}^*:=\frac{Np^-}{N-p^-}$;

\noindent \textbf{(iv)} we have
\[
  \limsup\limits_{|t|\rightarrow \infty} \frac{j(x,t)}{|t|^{p(x)}}< 0,
\]
 uniformly for almost all $x \in \Omega$;

\noindent \textbf{(v)} there exists $\mu >0$, such that
\[
  \limsup\limits_{|t| \rightarrow 0} \frac{j(x,t)}{|t|^{p(x)}} \leqslant -\mu,
\]
uniformly for almost all $x \in \Omega$;

\noindent \textbf{(vi)} there exists $\overline{u} \in W^{1,p(x)}_0 (\Omega) \setminus \{0\}$, such that
\[
  \frac{1}{p^-}\int_{\Omega}|\nabla \overline{u} (x)|^{p(x)}dx
  +\frac{\lambda_-}{p^-}\int_{\Omega}|\overline{u}(x)|^{p(x)}dx \leqslant \int_{\Omega} j(x, \overline{u}(x))dx,
\]
  where $\lambda_-:=\max\{0, -\lambda\}$.

\bigskip

\begin{remark}
   Hypothesis $H(j)(vi)$ can be replaced by 

\noindent \textbf{(vi')}
  there exists $\overline{u} \in W^{1,p(x)}_0 (\Omega) \setminus \{0\}$, such that
\[
  \overline{c} \|\overline{u}\|^{p^+} \leqslant \int_{\Omega} j(x, \overline{u}(x))dx,
  \qquad \textrm{ if } \; \|\overline{u}\| \geqslant 1,
\]
  or
\[
  \overline{c} \|\overline{u}\|^{p^-} \leqslant \int_{\Omega} j(x, \overline{u}(x))dx,
  \qquad \textrm{ if } \; \|\overline{u}\| < 1,
\]
\end{remark}
where $\overline{c}:=\max\{\frac{1}{p^-},\frac{\lambda_-}{p^-}\}$.

  We introduce locally Lipschitz functional $R: W_0^{1,p(x)}(\Omega) \rightarrow \mathbb{R}$ defined by
\[
  R(u)=\int_{\Omega} \frac{1}{p(x)}|\nabla u(x)|^{p(x)}dx -\int_{\Omega}  
  \frac{\lambda}{p(x)}|u(x)|^{p(x)}dx - \int_{\Omega} j(x,u(x))dx,
\]
  for all $u \in W_0^{1,p(x)}(\Omega)$. 

\begin{remark}
  The existence of nontrival solution for problem (P) was also considered in
  paper Barna\'s \cite{barnas}. In contrast to the last paper, instead of linear
  growth in $H(j)(iii)$ we assume the so-called sub-critical growth condition.   
  Moreover, condition $H(j)(v)$ is more general.
\end{remark}

\begin{lemma}\label{PS}
  If hypotheses $H(j)$ hold and $\lambda \in (-\infty, \frac{p^-}{p^+}\lambda_*)$, then $R$ satisfies the nonsmooth PS-condition.
\end{lemma}

\begin{proof}

  Let $\{u_n\}_{n \geq 1} \subseteq W_0^{1,p(x)}(\Omega)$ be a sequence such that $\{R(u_n)\}_{n \geq 1}$ is bounded and $m(u_n) \rightarrow 0$ as $n \rightarrow \infty.$
  We will show that $\{u_n\}_{n \geq 1} \subseteq W_0^{1,p(x)}(\Omega)$ is bounded.

  Because $|R(u_n)|\leq M$ for all $n \geq 1$, we have
\begin{equation} \label{gl}
  \int_{\Omega} \frac{1}{p(x)} |\nabla u_n (x)|^{p(x)} dx-\int_{\Omega} 
  \frac{\lambda}{p(x)}|u_n (x)|^{p(x)}dx - \int_{\Omega} j(x,u_n (x))dx \leq M.
\end{equation}
  So we obtain
\begin{equation}\label{330}
  \int_{\Omega} \frac{1}{p^+}|\nabla u_n (x)|^{p(x)} dx-\int_{\Omega}
  \frac{\lambda_+}{p^-}|u_n (x)|^{p(x)}dx -\int_{\Omega} j(x,u_n (x))dx \leq M,
\end{equation}
  where $\lambda_+:=\max\{\lambda,0\}$.

  From the definition of $\lambda_*$, we have
\begin{equation} \label{lambdaa}
  \lambda_* \int_{\Omega} |u_n(x)|^{p(x)}dx \leq \int_{\Omega} |\nabla u_n(x)|^{p(x)} dx,
\end{equation}
  for all $n \geq 1$.

  Using ($\ref{lambdaa}$) in ($\ref{330}$), we get
\begin{equation}\label{wlozeniee}
  \Big( \frac{\lambda_*}{p^+} - \frac{\lambda_+}{ p^-} \Big) \int_{\Omega}|
  u_n(x)|^{p(x)}dx - \int_{\Omega} j(x,u_n(x))dx \leq M.
\end{equation}

  By virtue of hypotheses $H(j)(iv)$, we know that  
\[
  \limsup\limits_{|t| \rightarrow \infty} \frac{j(x,t)}{|t|^{p(x)}}< 0,
\]
  uniformly for almost all $x \in \Omega$. So we can find $L>0$, such that for almost all $x \in \Omega$, all $t$ such that $|t| \geq L$, we have
\begin{equation}\label{353}
  \frac{j(x,t)}{|t|^{p(x)}}< -c<0.
\end{equation}
  It immediately follows that
\[
  j(x,t) \leq -c|t|^{p(x)} \quad \textrm{for all } t \textrm{ such that } |t|\geqslant L.
\]

  On the other hand, from the Lebourg mean value theorem (see Clarke \cite{Clarke}), for almost al $x \in \Omega$ and all $t \in \mathbb{R}$, we  can find $v(x) \in \partial j(x, k u(x))$ with $0<k<1$, such that
\[
  |j(x,t)-j(x,0)| \leq |v(x)||t|.
\]

  So from hypothesis $H(j)(iii)$, for almost all $x \in \Omega$
\begin{equation} \label{210}
  |j(x,t)| \leq a(x)|t|+c_1|t|^{r(x)} \leq a(x)|t|+c_1|t|^{r^+}+c_2,
\end{equation}
  for some $c_1, c_2>0$. Then for almost all $x \in \Omega$ and all $t$ such that $|t|<L$, from (\ref{210}) it follows that
\begin{equation} \label{123}
  |j(x,t)| \leq c_3,
\end{equation}
  for some $c_3>0$. Therefore, it follows that for almost all $x \in \Omega$ and all $ t \in \mathbb{R}$, we have
\begin{equation} \label{204}
  j(x,t) \leq -c|t|^{p(x)}+\beta,
\end{equation}
  where $\beta>0.$

  We use (\ref{204}) in (\ref{wlozeniee}) and obtain
\[
  \Big( \frac{\lambda_*}{p^+} - \frac{\lambda_+}{ p^-} \Big) \int_{\Omega}| u_n(x)|^{p(x)}dx \leq M - \int_{\Omega} (c|u_n(x)|^{p(x)}-\beta)dx,
\]
  for all $n \geq 1$, which leads to
\begin{equation}\label{2012}
  \Big( \frac{\lambda_*}{p^+} - \frac{\lambda_+}{ p^-}+c \Big)\int_{\Omega}|
  u_n(x)|^{p(x)}dx \leq M_1 \quad \forall n\geq 1,
\end{equation}
  for some $c, M_1>0$.

  We know that $\frac{\lambda_*}{p^+} - \frac{\lambda_+}{ p^-}+c>0$, so
\begin{equation}\label{ogr}
  \textrm{the sequence } \{u_n\}_{n \geq 1} \subseteq L^{p(x)} (\Omega) \textrm{ is bounded}
\end{equation}
  (see Lemmata \ref{lemma2} (c) and (d)).\\

\bigskip

  Now, let us consider two cases.

\noindent \textit{Case $1$. }

  Suppose that $\lambda \leqslant 0$.

  In this case, from (\ref{330}), we get
\[
  \frac{1}{p^+} \int_{\Omega} |\nabla u_n|^{p(x)} dx- \int_{\Omega} j(x,u_n(x))dx \leq M.
\]
  Using (\ref{204}) and fact that $ \{u_n\}_{n \geq 1} \subseteq L^{p(x)} (\Omega)$ is bounded, we obtain
\[
  \frac{1}{p^+} \int_{\Omega} |\nabla u_n|^{p(x)} dx \leq M_2,
\]
 for some $M_2>0.$
 So, we have that
\begin{equation}\label{ogr1}
  \textrm{the sequence } \{\nabla u_n\}_{n \geq 1} \subseteq L^{p(x)}  (\Omega;\mathbb{R}^N) \textrm{ is bounded}
\end{equation}
  (see Lemmata \ref{lemma5} (c) and (d)).\\

\noindent \textit{Case $2$. }

  Suppose that $\lambda > 0$.

  Similar to the first part of our proof, by using again (\ref{lambdaa}) in (\ref{330}) in another way, we obtain
\[
  \Big( \frac{1}{p^+} - \frac{\lambda_+}{\lambda_*p^-} \Big) \int_{\Omega} |\nabla u_n|^{p(x)} dx- \int_{\Omega} j(x,u_n(x))dx \leq M.
\]

  In a similar way, from (\ref{204}) and fact that $ \{u_n\}_{n \geq 1} \subseteq L^{p(x)} (\Omega)$ is bounded, we obtain
\[
  \Big( \frac{1}{p^+} - \frac{\lambda_+}{\lambda_*p^-}+c\Big) \int_{\Omega} |\nabla u_n|^{p(x)} dx \leq M_3,
\]
  for some $M_3, c>0.$ From fact that $\frac{1}{p^+} - \frac{\lambda_+}{\lambda_*p^-}>0$, we have that
\begin{equation}\label{ogr3}
  \textrm{the sequence } \{\nabla u_n\}_{n \geq 1} \subseteq L^{p(x)} (\Omega;\mathbb{R}^N) \textrm{ is bounded.}
\end{equation}

  From (\ref{ogr}), (\ref{ogr1}) and (\ref{ogr3}), we have that
\begin{equation}\label{ogr2}
  \textrm{the sequence } \{u_n\}_{n \geq 1} \subseteq W_0^{1,p(x)} (\Omega)   \textrm{ is bounded}
\end{equation}
(see Lemmata \ref{lemma5} (c) and (d)).\\

  Hence, by passing to a subsequence if necessary, we may assume that
\begin{equation}\label{1}
  \left.
    \begin{array}{ll}
      u_n \rightarrow u, & \textrm{weakly in } W_0^{1,p(x)} (\Omega),\\
      u_n \rightarrow u, & \textrm{in } L^{r(x)}(\Omega),
    \end{array}
  \right.
\end{equation}
  for any $r \in \mathcal{C}(\overline{\Omega})$, with $r^+=\max\limits_{x \in \Omega} r(x)< {\widehat{p}}^*:=\frac{Np^-}{N-p^-}.$ 

  Since $\partial R(u_n) \subseteq (W_0^{1,p(x)}(\Omega))^*$ is weakly compact, nonempty and the norm functional is weakly lower semicontinuous in a Banach space, then we can find $u_n^* \in \partial R(u_n)$ such that $||u_n^*||_*=m(u_n)$, for $n \geq 1$.

  Consider the operator $A:W_0^{1,p(x)}(\Omega) \rightarrow (W_0^{1,p(x)}(\Omega))^*$ defined by (\ref{wazne}).

  Then, for every $n \geq 1$, we have
\begin{equation}\label{11}
u_n^*=Au_n-\lambda |u_n|^{p(x)-2} u_n - v_n^*,
\end{equation}
  where $v_n^* \in \partial \psi (u_n)\subseteq L^{p'(x)} (\Omega)$, for $n \geq 1$, with $\frac{1}{p(x)}+\frac{1}{p'(x)}=1$.\\ $\psi: W_0^{1,p(x)}(\Omega) \rightarrow \mathbb{R}$ is defined by $\psi (u_n)=\int\limits_{\Omega} j(x,u_n(x)) dx$. We know that, if $v_n^* \in \partial \psi (u_n)$, then $v_n^*(x) \in \partial j(x, u_n(x))$ (see Clarke \cite{Clarke}). 

  From the choice of the sequence $\{u_n\}_{n \geq 1} \subseteq W_0^{1,p(x)}(\Omega)$, at least for a subsequence, we have
\begin{equation}\label{489}
  |\langle u_n^*,w \rangle| \leq \varepsilon_n \quad \textrm{for all } w \in W^{1,p(x)}_0(\Omega),
\end{equation}
  with $\varepsilon_n \searrow 0$.

  Putting $w=u_n-u$ in (\ref{489}) and using (\ref{11}), we obtain
\begin{equation} \label{47}
  \Big|\langle Au_n,u_n-u\rangle - \lambda \int_{\Omega} 
  |u_n(x)|^{p(x)-2}u_n(x)(u_n-u)(x)dx -\int_{\Omega} v_n^*(x) 
  (u_n-u)(x)dx\Big|\leq \varepsilon_n,
\end{equation}
  with $\varepsilon_n \searrow 0$.

  Using Lemma \ref{fan}(e), we see that
\begin{eqnarray*}
  &           & \lambda \int_{\Omega} |u_n(x)|^{p(x)-2}u_n(x)(u_n-u)(x)dx\cr
  & \leqslant & \lambda \Big(\frac{1}{p^-}+\frac{1}{p'^-}\Big) \|\, |u_n|^{p(x)-1}\|_{p'(x)}\|u_n-u\|_{p(x)},
\end{eqnarray*}
   where $\frac{1}{p(x)}+\frac{1}{p'(x)}=1$.

  We know that $\{u_n\}_{n \geqslant 1} \subseteq L^{p(x)}(\Omega)$
  is bounded, so using (\ref{1}), we conclude that
\[
  \lambda \int_{\Omega} |u_n(x)|^{p(x)-2}u_n(x)(u_n-u)(x)dx \rightarrow 0 \quad \textrm{as } n \rightarrow \infty
\]
  and
\[
  \int_{\Omega} v_n^*(x) (u_n-u)(x)dx \rightarrow 0 \quad \textrm{as } n \rightarrow \infty.
\]
 If we pass to the limit
as $n \rightarrow \infty$ in (\ref{47}), we have
\begin{equation} \label{68}
  \limsup\limits_{n \rightarrow \infty} \langle Au_n, u_n-u \rangle \leq 0.
\end{equation}

  So from Lemma \ref{lemma3}, we have that $u_n \rightarrow u$ in $W^{1,p(x)}_0(\Omega)$ as $ n \rightarrow \infty$. Thus $R$ satisfies the PS-condition.
\end{proof}

\begin{lemma} \label{lem3}
  If hypotheses $H(j)$ hold and $\lambda<\frac{p^-}{p^+}\lambda_*$,
  then there exist $\beta_1, \beta_2 >0$ such that for all
  $u \in W_0^{1,p(x)}(\Omega)$ with $\|u\|<1$, we have
\[
  R(u) \geqslant \beta_1 \|u\|^{p^+} - \beta_2 \|u\|^\theta,
\]
  with $p^+<\theta \leqslant \widehat{p}^{*}:=\frac{Np^-}{N-p^-}$.
\end{lemma}

\begin{proof}
  From hyphothesis $H(j)(v)$, we can find $\delta >0$,
  such that for almost all $x \in \Omega$ and all $t$ such that $|t|\leqslant \delta$, we have
\[
  j(x,t) \leqslant \frac{-\mu}{2}|t|^{p(x)}.
\]
  On the other hand, from hypothesis $H(j)(iii)$, we know
  that for almost all $x \in \Omega$ and all $t$ such that $|t|>\delta$, we have
\[
 |j(x,t)| \leq a_1|t|+c_1|t|^{r(x)},
\]
  for some $a_1, c_1>0$.
  Thus for almost all $x \in \Omega$ and all $t \in \mathbb{R}$ we have
\begin{equation}\label{porow}
  j(x,t) \leqslant\frac{-\mu}{2}|t|^{p(x)}+\gamma |t|^\theta,
\end{equation}
  for some $\gamma>0$ and $p^+<\theta<\widehat{p}^*$.

\bigskip

  Let us consider two cases.\\

\noindent \textit{Case 1. }
  Suppose that $\lambda \leqslant 0$. 

  By using (\ref{porow}), we obtain that
\begin{eqnarray*}
   R(u)&    =    &\int_{\Omega} \frac{1}{p(x)} |\nabla u (x)|^{p(x)} dx-\int_{\Omega} \frac{\lambda}{p(x)}|u (x)|^{p(x)}dx - \int_{\Omega} j(x,u (x))dx\\
       &\geqslant& \int_{\Omega} \frac{1}{p^+} |\nabla u (x)|^{p(x)} dx + \int_{\Omega} \frac{\mu}{2}|u(x)|^{p(x)} dx - \gamma \int_{\Omega}|u(x)|^\theta dx.
\end{eqnarray*}

  So, we have
\[
  R(u) \geqslant \beta_1 \Big[\int_{\Omega}|\nabla u(x)|^{p(x)}dx +\int_{\Omega} |u(x)|^{p(x)} dx\Big]-\gamma\|u\|^\theta_\theta,
\]

  where $\beta_1:=\min\{ \frac{1}{p^+},\frac{\mu}{2} \}.$\\

\noindent \textit{Case 2. }
  Suppose that $\lambda \in (0, \frac{p^-}{p^+} \lambda_*)$.

  By using (\ref{porow}) and the Rayleigh quotient, we obtain that

\begin{eqnarray*}
   R(u)&    =    &\int_{\Omega} \frac{1}{p(x)} |\nabla u (x)|^{p(x)} dx-\int_{\Omega} \frac{\lambda}{p(x)}|u (x)|^{p(x)}dx - \int_{\Omega} j(x,u (x))dx\\
       &\geqslant& \int_{\Omega} \frac{1}{p^+} |\nabla u (x)|^{p(x)} dx-\int_{\Omega} \frac{\lambda}{p^-}|u(x)|^{p(x)}dx\\
       &         &\hspace{2.7cm}+ \int_{\Omega} \frac{\mu}{2}|u(x)|^{p(x)} dx - \gamma \int_{\Omega}|u(x)|^\theta dx\\
       &\geqslant& \frac{1}{p^+} \int_{\Omega} |\nabla u(x)|^{p(x)} dx + \frac{\mu}{2} \int_{\Omega} |u(x)|^{p(x)}dx\\
       &         &\hspace{2.7cm} -  \frac{\lambda}{\lambda_*p^-}\int_{\Omega} |\nabla u(x)|^{p(x)} dx -\gamma \|u\|^\theta_\theta\\
       &    =    & \Big(\frac{1}{p^+}-\frac{\lambda}{\lambda_*p^-}\Big)\int_{\Omega} |\nabla u (x)|^{p(x)} dx + \frac{\mu}{2} \int_{\Omega} | u(x)|^{p(x)}dx-\gamma \|u\|^\theta_\theta.
\end{eqnarray*}
  In our case, we have
\[
 \frac{1}{p^+}-\frac{\lambda}{\lambda_*p^-} > 0,
\]
  so
\[
  R(u)\geqslant \beta_1 \Big[\int_{\Omega}|\nabla u(x)|^{p(x)}dx +\int_{\Omega} |u(x)|^{p(x)} dx\Big]-\gamma\|u\|^\theta_\theta,
\]
  where $\beta_1:=\min\{ \frac{1}{p^+}-\frac{\lambda}{\lambda_*p^-},\frac{\mu}{2} \}.$

  As $\theta \leqslant p^*(x)=\frac{Np(x)}{N-p(x)}$, then $W^{1,p(x)}_0(\Omega)$
  is embedded continuously in $L^\theta(\Omega)$
  (see Lemma \ref{fan}(c)).
  So there exists $c>0$ such that
\begin{equation}\label{48}
  \|u\|_\theta \leqslant c\|u\| \qquad \textrm{for all } u \in W ^{1,p(x)}_0(\Omega).
\end{equation}
  Using (\ref{48}) and Lemma \ref{lemma5}(d), for all $u \in W ^{1,p(x)}_0(\Omega)$ with $\|u\|<1$, we have
\[
  R(u) \geqslant \beta_1 \|u\|^{p^+}- \beta_2\|u\|^\theta,
\]
  where $\beta_2=\gamma c^{\theta}$.
\end{proof}

\noindent Using Lemmata \ref{PS} and \ref{lem3}, we can prove the following existence theorem for problem (P).

\begin{theorem}
  If hypotheses $H(j)$ hold and $\lambda<\frac{p^-}{p^+}\lambda_*$, then problem (P) has a nontrival solution.
\end{theorem}

\begin{proof}
  From Lemma \ref{lem3} we know that there exist $\beta_1,\beta_2>0$,
  such that for all $u\in W_0^{1,p(x)} (\Omega)$ with $\|u\|<1$, we have
\[
  R(u) \geqslant \beta_1 \|u\|^{p^+}- \beta_2\|u\|^\theta
  = \beta_1\|u\|^{p^+} \Big( 1-\frac{\beta_2}{\beta_1}\|u\|^{\theta - p^+}\Big).
\]
  Since $p^+ <\theta$, if we choose $\rho >0$ small enough,
  we will have that $R(u)\geqslant L>0$, for all $u \in W_0^{1,p(x)}(\Omega)$,
  with $\|u\|=\rho$ and some $L>0$.

  Now, let $\overline{u} \in W^{1,p(x)}_0 (\Omega)$. We have

\begin{eqnarray*}
  R(\overline{u})&    =    &\int_{\Omega} \frac{1}{p(x)} |\nabla \overline{u} (x)|^{p(x)} dx-\int_{\Omega}
     \frac{\lambda}{p(x)}|\overline{u} (x)|^{p(x)}dx - \int_{\Omega} j(x,\overline{u} (x))dx\\
                 &\leqslant&\frac{1}{p^-} \int_{\Omega} |\nabla \overline{u} (x)|^{p(x)} dx +\frac{\lambda_-}{p^{-}}
     \int_{\Omega} |\overline{u}(x)|^{p(x)} dx-\int_{\Omega} j(x,\overline{u} (x))dx,\\
\end{eqnarray*}
  where $\lambda_-:=\max\{0, - \lambda\}$.

  From hyphothesis $H(j)(v)$, we get
  $R(\overline{u}) \leqslant 0$.
  This permits the use of Theorem \ref{twierdzenie}
  which gives us $u \in W_0^{1,p(x)}(\Omega)$ such that
  $R(u)>0 = R(0)$ and $0 \in \partial R(u)$.
  From the last inclusion we obtain
\[
  0=Au-\lambda |u|^{p(x)-2}u-v^*,
\]
  where $v^* \in \partial \psi(u).$ Hence
\[
  Au=\lambda |u|^{p(x)-2}u+v^*,
\]
  so for all $v \in \mathcal{C}_0^\infty (\Omega)$, we have $\langle Au,v \rangle = \lambda \langle |u|^{p(x)-2}u,v\rangle +\langle v^*,v \rangle$.

  So we have
\begin{eqnarray*}
  &   & \int_{\Omega} |\nabla u(x)|^{p(x)-2}(\nabla u(x), \nabla v(x))_{\mathbb{R}^N}dx\cr
  & = & \int_{\Omega}\lambda |u(x)|^{p(x)-2}u(x)v(x)dx+\int_{\Omega}v^*(x)v(x) dx,
\end{eqnarray*}
  for all $v \in \mathcal{C}_0^\infty(\Omega)$.

  From the definition of the distributional derivative we have
\begin{equation}
  \left\{
  \begin{array}{lr}
    -\textrm{div} \big( |\nabla u(x)|^{p(x)-2} \nabla u(x) \big)=\lambda |u(x)|^{p(x)-2} u(x) +v(x)& \textrm{in } \Omega,\\
    u=0 & \textrm{on}\ \partial \Omega,
  \end{array}
  \right.
\end{equation}
  so
\begin{equation}
  \left\{
  \begin{array}{lr}
    -\Delta_{p(x)}u-\lambda |u(x)|^{p(x)-2} u(x)\in \partial j(x, u(x))& \textrm{in } \Omega,\\
    u=0 & \textrm{on}\ \partial \Omega.
  \end{array}
  \right.
\end{equation}
  Therefore $u \in W_0^{1,p(x)}(\Omega)$ is a nontrivial solution of (P).
\end{proof}

\begin{remark}
  A nonsmooth potential satisfying hypothesis $H(j)$ is for example the one given by the following function:

\[
  j(x,t)= 
   \left\{
     \begin{array}{lcc}
      -\mu|t|^{p(x)} & \textrm{if} & |t| \leqslant 1,\\
      (\mu+\sigma -|2|^{p(x)})|t|-2 \mu-\sigma+|2|^{p(x)} & \textrm{if} & 1<|t|\leqslant 2,\\
      \sigma-|t|^{p(x)} & \textrm{if} & |t| > 2,\\
     \end{array}
    \right.
\]
  with $\mu, \sigma>0$ and continuous function $p:\overline{\Omega} \rightarrow \mathbb{R}$ which satisfies $1<p^- \leqslant p(x) \leqslant p^+<N<\infty $
 and $ p^+\leqslant \widehat{p}^*.$

\end{remark}

\end{document}